\documentclass{elsarticle}
\usepackage{amssymb}
\makeatletter
\def\ps@pprintTitle{%
 \let\@oddhead\@empty
 \let\@evenhead\@empty
 \def\@oddfoot{}%
 \let\@evenfoot\@oddfoot}
\makeatother
\usepackage{amssymb,amsthm,amsmath,mathtools}
\usepackage{ifthen}
\usepackage{graphicx}
\usepackage{graphics}
\usepackage{pdfpages}
\usepackage{float}
\usepackage{subfigure}
\usepackage{subcaption}
\usepackage{multirow}
\usepackage{wrapfig}
\usepackage{enumitem}
\usepackage{float}
\usepackage{multicol}
\usepackage{graphics}
\usepackage{multicol}
\DeclareUnicodeCharacter{2064}{}
\DeclareUnicodeCharacter{0308}{}
\usepackage[active]{srcltx}
\usepackage{placeins}
\usepackage[T1]{fontenc}
\usepackage[utf8]{inputenc}
\usepackage{caption}
\usepackage{mathrsfs}
\usepackage{algorithm}
\usepackage{physics}
\usepackage{tabularx,ragged2e,booktabs}
\newcolumntype{C}[1]{>{\Centering}m{#1}}

\newtheorem{ex}{Example}[section]
\newtheorem{lemma}{Lemma}[section]
\newtheorem{thm}{Theorem}[section]
\newtheorem{rmk}{Remark}[section]

\usepackage{mathtools}
\newtheoremstyle{case}{}{}{}{}{}{:}{ }{}
\theoremstyle{case}
\newtheorem{case}{\textbf{Case}}
\numberwithin{equation}{section}
\usepackage{geometry}
\begin{document}

\begin{frontmatter}
\title{Fractional Quadrature rule and using its Exactness for the M\"untz-Legendre Scaling Functions for Solving Fractional Differential Equations}

\author{Ritu Kumari$^{a}$}

\author{Mani Mehra$^{a*}$}
\ead{mmehra@maths.iitd.ac.in}
\author{Abhishek Kumar Singh$^{a,b}$}
\address{$^{a}$Department of Mathematics, Indian Institute of Technology Delhi\\
	New Delhi-110016, India}
\address{$^{b}$Institute of Mathematics and Computer Science, Universit\"at Greifswald,
	Walther-Rathenau-Stra{\ss}e~47, 17489 Greifswald, Germany}

\cortext[cor1]{Corresponding author}



\begin{abstract}
Fractional operators (derivatives/integrals) are defined via the integration of the functions. When the function is produced by a spanning set of fractional power functions, traditional quadrature rules often need to be revised, failing to provide exact evaluations for fractional power functions and thus introducing approximation errors. In this paper, we have formulated a fractional quadrature rule that achieves exact integration for functions within this specific set to address this issue. Some properties of the fractional quadrature rule have been proved, and the absolute error bound in the proposed fractional quadrature rule has been derived. The behavior of roots of the orthogonal M\"untz polynomial has also been observed for its application as nodes in the fractional quadrature rule. To illustrate the effectiveness of the newly proposed fractional quadrature rule, we focus on fractional differential equations that incorporate the left Caputo fractional derivative. In this context, M\"untz-Legendre scaling functions are utilized to approximate the Caputo derivative of functions involved in these equations. Additionally, we have derived an operational matrix for Riemann-Liouville integration to approximate the respective functions with the help of the fractional quadrature rule. To demonstrate the practical utility of our method, we provide illustrative examples that compare the $L_2$-error estimates in the solutions of fractional differential equations using our approach against those obtained with the Block-pulse method. These comparisons underscore the superior accuracy of our proposed method.
    
\end{abstract}
\end{frontmatter}

\section{\textbf{Introduction}}\label{8s1}
Fractional calculus is an extension of classical calculus to the non-integer order. The mathematical model involving non-integer order derivatives/integrals can better represent physical phenomena than the integer-order mathematical model. Its usefulness is demonstrated by applications in control theory \cite{kumar2021collocation} and viscoelasticity \cite{MERAL2010939}, among other scientific and technical fields. ⁤Moreover, the significance of fractional calculus in specific fields such as fluid mechanics \cite{kulish2002application}, speech signal modelling \cite{assaleh2007modeling}, sound wave transmission in rigid porous media \cite{fellah2002application}, edge detection in image processing \cite{mathieu2003fractional}, and the study of ultrasonic wave propagation in human cancellous bone \cite{sebaa2006application} adds to its adaptability. ⁤⁤Additionally, it provides a unique mathematical model for diffusion equations on metric graphs \cite{mehandiratta2019existence}, exhibiting its broad applicability and potential in various research fields.
⁤\par     
Utilizing orthogonal polynomials for approximation is a fundamental technique in approximation theory, enabling the transformation of complex problems into more manageable algebraic equations. This method is particularly effective for solutions that exhibit sufficient smoothness, with extensive research dedicated to studying the convergence of such approximations by various researchers \cite{canuto2007spectral1,canuto2007spectral2,guo2006optimal,gottlieb1977numerical,shen2011spectral}.

However, researchers often encounter discontinuous functions in many situations, and using continuous polynomial functions to approximate such functions poses a challenge as the resulting approximation is continuous. Thus, convergence is very slow in this case. Numerous researchers have used piecewise smooth functions, which can represent the discontinuities in non-smooth solutions of differential equations accurately \cite{pedas2012piecewise,rahimkhani2018muntz}. In this work, M\"untz-Legendre scaling functions, which generalize the concept of Legendre scaling functions, have been utilized. Moreover, M\"untz polynomials, which are an alternative to classical integer-order polynomials, offer a more accurate approximation for fractional power functions and have been applied in addressing complex problems such as variable-order stochastic fractional integro-differential equations \cite{singh2023algorithm}.  \par 
\subsection{Motivation and Contribution of the Paper}
In this paper, we address the challenge of computing integrals that are not expressible in terms of elementary functions or not solvable by direct integration formulas, mainly when dealing with piecewise fractional polynomials with singular weight functions. To overcome this challenge, researchers often use two approaches; the first is using some other functions like Block-pulse \cite{kumar2023muntz} and Hat functions \cite{tripathi2013new} to evaluate the integration. While using these functions, an additional error is produced in the approximation. Therefore, the first approach always gives the additional approximation error unless the function belongs to the same class. The second approach approximates the integration using quadrature (or cubature) rules. While these rules provide exact results within the class of classical polynomials, they face limitations when applied to integrals involving fractional polynomials due to the singular kernel and support associated with such functions.\par 
In this work, we deploy the piecewise fractional polynomials to solve the fractional differential equations (FDEs) because the fractional derivative and integration of the fractional polynomials are again the fractional polynomials. However, due to the limit of the integration and singular kernel and support of the piecewise fractional polynomials, we cannot find the exact results for integration. In the case of fractional polynomials integrands, the classical quadrature formula \cite{shen2011spectral} does not give exact results. Table \ref{tab:litrev} lists the relevant works, and we can observe that these methods introduce an additional error in the approximation. Therefore, this paper introduces a novel fractional quadrature rule tailored specifically for fractional integrands, aiming to achieve precise results. By deploying this fractional quadrature rule, we offer a solution for solving fractional differential equations with improved accuracy. Our contribution lies in developing this fractional quadrature rule, which fills a crucial gap in existing methodologies for handling fractional integrands, enabling more precise solutions in the realm of fractional differential equations. We will discuss the contribution of the paper in the following paragraph.\par 
In this paper, we consider the following FDEs 
\begin{equation} \label{eq1prblm}
    \begin{aligned}
        {}_0^C D_t^{\alpha} y(t) &= f(t,y(t)), ~~~t\in (t_0,t_f],\\
        y(0) &= y_{0},
    \end{aligned}
\end{equation}
where ${}_0^C D_t^{\alpha}$ is the left Caputo fractional derivative of order $\alpha~ \in (0,1]$. The solution of the FDEs is approximated using the M\"untz-Legendre scaling functions.
The main contributions of this paper are:

\begin{itemize}
\item This paper aids a novel fractional quadrature rule that exactly integrates M\"untz polynomials to a specific degree.
    \item 
 The derivation of the Riemann-Liouville integration operational matrix in previous research (as listed in Table \ref{tab:litrev}) has been found to introduce additional errors impacting the convergence rate towards the exact solution. Therefore, the proposed fractional quadrature rule is utilized to construct the operational matrix for M\"untz-Legendre scaling functions.
\item Exactness and error bound of the proposed fractional quadrature rule have been established. Moreover, the behaviour of nodes taken in the fractional quadrature rule have been investigated to get better approximations utilizing the proposed
rule.
\end{itemize}
\begin{table}[!hp]
    \centering
    \footnotesize
    {\begin{tabular}{|p{0.7cm}| p{3.2cm}| p{6.5cm}|}
         \hline
         \textbf{Year} &  \textbf{Author} &  \textbf{Results}\\ [1ex] \hline 
         2023 & Kumar and Mehra \cite{kumar2023muntz} & Proposed M\"untz-Legendre wavelet method utilizing Block-pulse functions to construct the operational matrix for fractional integration for solving Sturm-Liouville fractional optimal control problem.\\   \hline
2021 & Nemati and Torres \cite{articleliterature1} & Proposed a spectral method based on generalized and modified Hat functions for solving systems of fractional differential equations. \\ 
\hline
2021 & Maleknejad et al. \cite{maleknejad2021numerical} &Proposed a M\"untz-Legendre wavelet approach using Laplace transform for the numerical solution of distributed order fractional differential equations. \\ 
\hline
2017 & Rahimkhani et al. \cite{rahimkhani2018muntz} & Proposed M\"untz-Legendre wavelet method via piecewise fractional power functions to derive operational matrix 
of fractional-order integration 
for solving the fractional pantograph differential
equations. \\   
\hline

2015 & Bhrawy  et al. \cite{bhrawy2015review} & Reviewed the orthogonal polynomials  method to derive the operational matrix 
of fractional-order derivatives for approximation. \\ 
 \hline
    \end{tabular}}
     \caption{\textcolor{black}{Comprehensive literature review of some recent research in related domain.}}
     \label{tab:litrev}
\end{table}

\par
The organization of the paper is as follows: Section \ref{8s2} introduces foundational concepts in fractional calculus and the theory behind M\"untz-Legendre scaling functions, along with how these functions can be utilized for approximating other functions. Section \ref{section3} presents a newly developed fractional quadrature rule, detailing its characteristics, such as its precision with M\"untz polynomials to a specified degree and the benefits of selecting the orthogonal M\"untz polynomial roots as quadrature nodes. This section also validates the error limits for the introduced fractional quadrature rule. In Section \ref{section4}, the operational matrix for left Riemann-Liouville integration concerning M\"untz-Legendre scaling functions is derived, facilitated by the fractional quadrature rule and a transformation matrix that links the M\"untz-Legendre scaling functions basis to piecewise fractional power functions \cite{rahimkhani2018muntz}. Section \ref{section5} outlines the numerical method employed for solving FDEs, supplemented by test examples that compare the $L_2$ error outcomes of our method against those obtained through the Block-pulse technique. The paper concludes with Section \ref{section6}, summarizing the findings and contributions of the study.

\section{\textbf{Preliminaries}}\label{8s2}
In this section, some of the basic definitions of fractional derivatives and integrals are discussed. Subsequently, we explore constructing the M\"untz-Legendre scaling function by dilating and translating a set of M\"untz-Legendre polynomials. Further, the approximation of a square-integrable function using these scaling functions is explored. Now, we discuss some basic definitions of fractional calculus, laying the groundwork for the subsequent development of the paper in later sections.\par
For a given function $p(t),~t \in [a,b]$, the left Riemann-Liouville fractional integral of order $\alpha > 0$, denoted by $_{a}I_{t}^{\alpha}$, is defined as follows \cite{li2015numerical}
	\begin{equation}
		\begin{aligned}
			_{a}I_{t}^{\alpha}p(t) &= \frac{1}{\Gamma(\alpha)}\int_{a}^{t}(t-z)^{\alpha-1}p(z)dz, \\
			_{a}I_{t}^{0}p(t) &= p(t).
		\end{aligned}
	\end{equation}
The left Riemann-Liouville fractional derivative of order $\alpha > 0$, denoted by $_a^{RL}D_{t}^{\alpha}$, is defined as follows \cite{li2015numerical}
	\begin{equation}
		_a^{RL}D_{t}^{\alpha}p(t) = \frac{1}{\Gamma(m-\alpha)}\frac{d^{m}}{dt^{m}}\int_{a}^{t}(t-z)^{m-\alpha-1}p(z)dz,~m-1 \leq \alpha < m,
	\end{equation}
	where the function $p(t)$ has absolutely continuous derivative up to order $(m-1)$.
Further,  the left Caputo fractional derivative of order $\alpha > 0$, denoted by $_a^CD_{t}^{\alpha}$, is defined via the above Riemann-Liouville fractional derivative \cite{li2015numerical} as
	\begin{equation}
		_a^C D_{t}^{\alpha}p(t) =~ _a^{RL}D_{t}^{\alpha}\big[p(t)-\sum_{k=0}^{m-1}\frac{p^{(k)}(a)}{k!}(t-a)^{k}\big],~ m-1 < \alpha \leq m.
	\end{equation}
Moreover, when $m=1$, i.e., $\alpha=(0,1]$, the left Caputo fractional derivative of order $\alpha > 0$ is given by
\begin{equation}
		_a^C D_{t}^{\alpha}p(t) = \frac{1}{\Gamma(1-\alpha)}\int_a^t (t-z)^{-\alpha}p^{(1)}(z)dz.
	\end{equation}

\subsection{M\"untz-Legendre Wavelet}\label{2.2}
Applications of wavelets are found in various disciplines such as mathematics and engineering \cite{akansu2010emerging}. Many essential properties, such as compact support, orthogonality, and vanishing moments (the ability to exactly approximate polynomials up to a certain degree), are possessed by wavelets \cite{mehra2018wavelets}. Furthermore, in approximating non-smooth functions, they are found to be more efficient compared to orthogonal polynomials. Better convergence is achieved using wavelets even if the exact solution is not smooth, unlike in the case of orthogonal polynomials, where low-order algebraic convergence is obtained.
\par
The first generation of wavelets is achieved through the dilation and translation of a single function (or a set of functions) known as the mother wavelet \cite{mehra2018wavelets}. The introduction of M\"untz-Legendre scaling functions has been made, followed by a brief exposition on function approximation using these functions. Due to the property that allows for better accuracy in approximating functions consisting of fractional powers, M\"untz-Legendre scaling functions are considered superior to other wavelets.
The M\"untz-Legendre polynomials defined in the interval $[0,1]$ is given by
\begin{align}\label{muntzpoly}
L_m(t)=\sum_{i=0}^m c_{i,m} t^{\lambda_i},
\end{align}
where
$$c_{0,0}=1, \ \ c_{i,m}=\frac{\prod_{j=0}^{m-1}(1+\lambda_i+\lambda_j)}{\prod_{\substack{j=0\\j\neq i}}^m(\lambda_i-\lambda_j)}.$$
Here $\lambda_i \neq \lambda_j$ for $i \neq j$ and $\lambda_i > \frac{1}{2}, \forall i$. These polynomials are orthogonal with respect to the unit weight function and satisfy the following orthogonality condition:
$$\int_0^1 L_k(t) L_m(t)dt=\begin{cases} \frac{1}{2\lambda_m+1}, \ \ &\text{if} \ k=m ,\\
0, \ \ &\text{otherwise}.
\end{cases}$$
The M\"untz-Legendre scaling functions \cite{kumar2023muntz} in the interval $[0,1)$ are defined as
\begin{align}
    \phi_{n,m}^J(t)=\begin{cases}
        2^\frac{J-1}{2} \sqrt{2\lambda_m + 1} L_m(2^{J-1}t - n + 1), \ \ &\text{if} \ \frac{n-1}{2^{J-1}} \leq t < \frac{n}{2^{J-1}},\\
        0, &\text{otherwise},
    \end{cases}
\end{align}
where $n = 1,2,\ldots2^{J-1}$ is the translation parameter, $J=1,2,\ldots$ is the dilation parameter and $m=0,1,2,\ldots,M-1$ is the degree of the M\"untz-Legendre polynomials \cite{hosseinpour2018new}.
Throughout this paper, it has been assumed that $\lambda_i = i\lambda, ~i=0,1,2,\ldots,M-1$, where $\lambda$ is a fixed real constant.

\subsection{Function Approximation}
It is known that any function $p\in L_2([0,1))$ can be expanded with M\"untz-Legendre scaling functions $\phi_{n,m}^J(t)$ as

\begin{align*}
    p(t)= \sum_{n=1}^\infty \sum_{m=0}^{M-1}p_{n,m}^J\phi_{n,m}^J(t),
\end{align*}
where $p_{n,m}^J= \int_0^1 p(t)\phi_{n,m}^J(t)dt$ and $J\rightarrow \infty$. For numerical implementations, such $p_{n,m}^J$ are found, so that the truncated expansion with 
scaling functions $\phi_{n,m}^J(t)$, 
\begin{align}\label{span}
    p(t) \approx \sum_{n=1}^{2^{J-1}} \sum_{m=0}^{M-1}p_{n,m}^J\phi_{n,m}^J(t) = P^T\Phi(t),
\end{align}
is the best approximation of $p(t)$ in $L_2$ sense, where 
\begin{equation}
    \begin{aligned} 
    \Phi(t)=&[\phi_{1,0}^J(t),\phi_{1,1}^J(t),\ldots,\phi_{1,M-1}^J(t),\phi_{2,0}^J(t),\ldots,\phi_{2,M-1}^J(t),\ldots,\phi_{2^{J-1},0}^J(t),\\
    &\ldots,\phi_{2^{J-1},M-1}^J(t)]^T,\\
    =&[\mathcal{\phi}_{1}(t),\mathcal{\phi}_2(t),\ldots,\mathcal{\phi}_M(t),\mathcal{\phi}_{M+1}(t),\ldots,\mathcal{\phi}_{2^{J-1}M}(t)]^T,
\end{aligned}
\end{equation}
and 
\begin{align*}
  P=[p_{1,0}^J,p_{1,1}^J,\ldots,p_{1,M-1}^J,p_{2,0}^J,\ldots,p_{2,M-1}^J,\ldots,p_{2^{J-1},0}^J,
  \ldots,p_{2^{J-1},M-1}^J]^T.  
\end{align*}
Thus, all the functions henceforth will be approximated using Equation (\ref{span}) in this paper.


The classical quadrature rule \cite{shen2011spectral} is used to approximate the integral of a function $p(t)$ with respect to some weight function $w(t)$ in an interval $[a,b]$. However, when $p(t)$ or $w(t)$ is a fractional power function, this quadrature rule does not yield good results. To address this issue, a new fractional quadrature rule is introduced, which will be utilized in further sections to approximate the integrals of fractional power functions with improved accuracy.
\section{Fractional Quadrature Rule}\label{section3}
Consider the definite integral 
\begin{align}
\int_a^b p(t) w(t) dt, \label{1}
\end{align}
where $p(t)$ is the function to be integrated with respect to the weight function $w(t)$ on the interval $(a,b)$.
A new fractional quadrature rule is proposed to approximate (\ref{1}), which is given by
\begin{align}\label{4.5}
 \int_a^b p(t) w(t) dt \approx \sum_{j=0}^N p(t_j)w_j,   
\end{align}
where
\begin{align}\label{weight}
    w_j=\int_a^b h_j(t) w(t) dt,
\end{align}
and
\begin{align}\label{hj}
    h_j(t)= \prod_{\substack{i=0\\i\neq j}}^N \frac{t^\lambda-t_i^\lambda}{t_j^\lambda-t_i^\lambda},
\end{align}
where  $\{t_j\}_{j = 0}^{N}$ are the quadrature nodes, $\{w_j\}_{j = 0}^{N}$ are the quadrature weights and $\lambda$ is the same parameter as used in the M\"untz-Legendre scaling functions in Section (\ref{2.2}). When $\lambda = 1$ in the Equation \eqref{hj}, then $h_{j}(t)$ becomes the Lagrange basis. Now, some properties of the proposed fractional quadrature rule will be proven.
\subsection{Properties of Fractional Quadrature Rule} \label{Properties of Fractional Quadrature Rule}
We can write 
\begin{align}\label{quderror}
\int_a^b p(t)w(t) dt = \sum_{j=0}^N p(t_j)w_j + R_N(p), 
\end{align}
where $R_N(p)$ is the quadrature error. It is said that the fractional quadrature rule is exact for M\"untz polynomials, $p(t)$, of degree less than or equal to $N$, if
$$\int_a^b p(t)w(t) dt = \sum_{j=0}^N p(t_j)w_j,$$
that is
$$R_N(p)=0.$$
Before proving the exactness results, some important lemmas are discussed to establish the exactness of the fractional quadrature rule (\ref{4.5}). Initially, a significant lemma is stated, demonstrating the existence and uniqueness of orthogonal monic M\"untz polynomials under any weight function.
\begin{lemma}\label{l2.2}
    Let $w>0$ be a given weight in $L^1(a,b)$. Then for any $n \geq 1$, we can find a unique sequence of orthogonal  monic M\"untz polynomials $\{p_n^\lambda\}$ using the recurrence relation
    \begin{align*}
        p_0^\lambda&=1,\\
        p_1^\lambda&=x^\lambda-\alpha_0p_0^\lambda,\\
        p_{n+1}^\lambda&=(x^\lambda-\alpha_n)p_n^\lambda-\beta_np_{n-1}^\lambda,
    \end{align*}
    where
    \begin{align*}
        &\alpha_n=\frac{(x^\lambda p_n^\lambda,p_n^\lambda)_w}{\norm{p_n^\lambda}_w^2};  \ \ n\geq0,~~~
         \beta_n=\frac{\norm{p_n^\lambda}_w^2}{\norm{p_{n-1}^\lambda}_w^2}; \ \ n\geq1. 
    \end{align*}
\end{lemma}
\begin{proof}
    See \cite{singh2024modified}.
\end{proof}
Next, an important lemma will be proved, which is crucial for establishing the exactness of the fractional quadrature rule (\ref{4.5}).
\begin{lemma}
Let $p(t)\in C^{N+1}[a,b]$, then for any set of quadrature nodes $\{t_j\}$ and weights $\{w_j\}, j=0,1,\ldots ,N$, where $w_j=h_j(t)$ as defined in (\ref{hj}),  the fractional interpolation formula is given by
\begin{align}\label{interpolation}
    p(t)\approx \sum_{j=0}^N p(t_j)w_j. 
\end{align}
If $e_N(t)$ is the interpolation error, that is,
$$p(t)=\sum_{j=0}^N p(t_j)w_j+e_N(t),$$
then $e_N(t)=0$ for all M\"untz polynomials, $p(t)$, of degree less than or equal to $N$.
\end{lemma}
\begin{proof}
    We have
    $$p(t)=\sum_{j=0}^N p(t_j)h_j(t)+e_N(t).$$
    Assuming $y=t^\lambda$,  $p(t)=P(t^{\lambda})$, $h_j(t)=H_j(t^{\lambda})$ and $e_N(t)=E_N(t^{\lambda})$, then
    $$P(y)=\sum_{j=0}^N p(t_j)H_j(y)+E_N(y),$$
    where 
    $$H_j(y)=\prod_{\substack{i=0\\i\neq j}}^N \frac{y-t_i^\lambda}{t_j^\lambda-t_i^\lambda}.$$
    Also, the error function $E_N(y)$ can be written as
    \begin{align}\label{4.9}
        E_N(y)=(y-t_0^\lambda)(y-t_1^\lambda)\ldots(y-t_N^\lambda)g(y).
    \end{align}
    Thus, the function $P(y)$ using Equation (\ref{4.9}) is rewritten as 
    \begin{align}\label{eqn}
         P(y)=\sum_{j=0}^N p(t_j)H_j(y)+(y-t_0^\lambda)(y-t_1^\lambda)\ldots(y-t_N^\lambda)g(y).
    \end{align}
    Constructing a function $k(x,y)$, 
    \begin{align}
        k(x,y)=P(x)-P_N(x)-(x-t_0^\lambda)(x-t_1^\lambda)\ldots(x-t_N^\lambda)g(y),
    \end{align}
    where 
    $$P_N(x)=\sum_{j=0}^N p(t_j)H_j(x).$$
    Since the function $k(x,y)$ has $N+2$ zeros (when considered as a function of $x$), viz., $t_0^\lambda, t_1^\lambda, \ldots, t_N^\lambda, y$,  and satisfy all the assumptions for applying Rolle's Theorem,  therefore there exist $N+1$ points in $[a,b]$, say $\{\zeta_r\}_{r=0}^N$ such that 
    $$k'(\zeta_r)=0.$$
   Similarly, by applying the theorem recursively, a point $\zeta \in [a,b]$ is found such that
    \begin{align*}
        k^{N+1}(\zeta)=0.
        \end{align*}
        Using this fact in (\ref{eqn}),we get
            \begin{align*}            
             P^{N+1}(\zeta)=(N+1)!g(y), \end{align*}
             where the derivative is taken with respect to $x$.\\
             As a result of the above calculations, the function $g(y)$ in the error function $E_N(y)$ can be expressed as
      \begin{align*} &g(y)=\frac{1}{(N+1)!} P^{N+1}(\zeta).
        \end{align*}
 Using the value of $g(y)$ in (\ref{4.9}), we get
 \begin{align}
E_N(y)=\frac{1}{(N+1)!}(y-t_0^\lambda)(y-t_1^\lambda)\ldots(y-t_N^\lambda)   P^{N+1}(\zeta).  
 \end{align}
 Now, any M\"untz polynomial of degree $n$ is of the form
 \begin{align}
     L_n(t)=a_nt^{n\lambda}+a_{n-1}t^{(n-1)\lambda}+\ldots+a_1t^\lambda+a_0.
 \end{align}
 If $p(t)=L_n(t)$=$P(t^\lambda)$, then for all $n\leq N$;
 $$P^{n+1}(\zeta)=0,$$
 and consequently
 $$E_N(t^\lambda)=e_N(t)=0.$$
Hence, it is proved that 
\begin{align} \label{4.14}
    p(t) = \sum_{j=0}^N p(t_j)w_j,  
\end{align} that is, the fractional interpolation formula (\ref{interpolation}) coincides with the function $p(t)$, if $p(t)$ is a M\"untz polynomial of degree less than or equal to $N$.
\end{proof}
In view of the above lemmas, the exactness of the fractional quadrature rule (\ref{4.5}) will now be proven.
\begin{thm}\label{exactness1}
 If $p(t)\in C^{N+1}[a,b]$, then for any set of quadrature nodes $\{t_j\}$ and weights $\{w_j\}$ as defined in (\ref{weight}) , $j=0,1,\ldots ,N$, the fractional quadrature rule 
(\ref{4.5})
is exact for M\"untz polynomials of degree less than or equal to $N$.   
\end{thm}
\begin{proof}
Using expression (\ref{4.14}),we have
\begin{align*}
p(t) = \sum_{j=0}^N p(t_j)h_j(t),    
\end{align*}
where $p(t)$ is a M\"untz polynomial of degree less than or equal $N$.\\
Integrating both sides with respect to weight $w(t)$, we get
\begin{align*}
    \int_a^b p(t)w(t) dt &= \sum_{j=0}^N\int_a^b p(t_j)h_j(t)w(t) dt\\
    &=\sum_{j=0}^Np(t_j)\int_a^b h_j(t)w(t) dt\\
    &=\sum_{j=0}^Np(t_j)w_j. 
\end{align*}
\end{proof}
Using Theorem \ref{exactness1}, a stronger exactness result of the proposed fractional quadrature rule (\ref{4.5}) can be proved for M\"untz polynomials of degree up to $2N+1$.
\begin{thm}\label{th 4.1}
    Let $\{t_j\}_{j=0}^N$ be the zeros of $p_{N+1}^\lambda$, where $p_{N+1}^\lambda$ is the orthogonal M\"untz polynomial of degree $N+1$ from the orthogonal sequence of M\"untz polynomials with respect to the weight function $w(t)$ as obtained in Lemma (\ref{l2.2}). Then, the fractional quadrature rule (\ref{4.5}) is exact for M\"untz polynomials upto degree $2N+1$, i.e.,
    \begin{align}\label{4.16}
     \int_a^b p(t)w(t) dt = \sum_{j=0}^N p(t_j)w_j,    
    \end{align}
    where $p(t)$ is a M\"untz polynomial of degree less than or equal to $2N+1$.
\end{thm}
\begin{proof} 
The set of M\"untz polynomials upto degree n can be represented as
$$M_n(a,b)=span\{t^{\lambda_0},t^{\lambda_1},\ldots,t^{\lambda_n}\}.$$
Firstly, the existence of $p_{N+1}^\lambda$ is ensured by the Lemma (\ref{l2.2}).
Again, treating the M\"untz polynomial as classical polynomial by assuming $y=t^\lambda$, we get $N+1$ roots of $p_{N+1}^\lambda(t)=0$ in $[a,b] $, say $t_0, t_1, \ldots, t_N$.
Moreover, by virtue of Division Algorithm on $\mathbb R[y]$ (set of all real polynomials over $\mathbb R$), for any $p(t) \in M_{2N+1}(a,b) $, there exists $q(t),r(t) \in M_N(a,b)$ such that
\begin{align}\label{division alg}
p(t)=q(t)p_{N+1}^\lambda(t)+r(t).
\end{align}
Since $p_{N+1}^\lambda(t_j)=0$ for $j=0,1,\ldots,N$, so
\begin{align*}
    p(t_j)=r(t_j).
\end{align*}
Since the degree of $q(t)$ is less than or equal to $N$, using the orthogonality of $p_{N+1}^\lambda$,
$$\int_a^b q(t) p_{N+1}^\lambda(t)w(t)=0.$$
Therefore, multiplying by the weight function $w(t)$ in Equation (\ref{division alg}) and integrating, we obtain
\begin{align*}
\int_a^b p(t)w(t) dt &=  \int_a^b r(t)w(t) dt \\
&= \sum_{j=0}^N r(t_j)w_j\\
&= \sum_{j=0}^N p(t_j)w_j,  \ \ \ \ \ \forall p\in M_{2n+1}(a,b).
\end{align*}
\end{proof}
The exactness of the fractional quadrature rule (\ref{4.5}) up to degree $2N+1$ has been proved, and now it will be verified by two test examples in the following example.
\begin{ex}
   Theorem (\ref{th 4.1}) is verified for M\"untz-Legendre Polynomial of degree $5$ and $7$, denoted by $L_5(t)$ and $L_7(t)$ respectively, as defined in (\ref{muntzpoly}).
\end{ex}
In this example, $\lambda=0.75$ is used. The integrals of $p(t)=L_5(t)$ and $p(t)=L_7(t)$ with respect to the weight function $w(t)=(1-t)$ are exactly approximated as in expression (\ref{4.16}) for $N=2$ and $3$, respectively. The computed error in approximation is of the order of $10^{-33}$ and $10^{-34}$, respectively, which can be considered as machine error. \par
Now, another important property of the fractional quadrature rule (\ref{4.5}) related to the node points $t_j$'s, where $t_j$'s are the roots of the orthogonal M\"untz polynomials as obtained in Lemma \ref{l2.2}, will be discussed. 
More precisely, the fractional quadrature rule (\ref{4.5}), with nodes being the roots of orthogonal M\"untz polynomials, $p_n^{\lambda}$ (orthogonal with respect to a weight function $w'(t)=(c-t)^p g(t),~g(c)\neq 0,~p \in \mathbb R^+$, where c is the point where $w(t)$ in the integral (\ref{4.5}) approaches zero), will be analyzed.
For instance, consider a function $p(t)$ which is to be integrated with respect to a weight function $w(t)$ in $[a,b]$ using the fractional quadrature rule (\ref{4.5}) and suppose $w(t)$ approaches zero at a point $c$ in $[a,b]$. Then, obviously, for better approximation, the distance of the nodes $t_j's$ in (\ref{4.5}) from the point $c$ should be larger, i.e., the approximation error $\propto$ $\frac{1}{\norm{ t-c}}_2$, where $t=(t_0,t_1,t_2,\ldots t_N)$. Some computational experiments were conducted, and it was found that $\mid{ t_j-c}\mid$ increases as the magnitude of $\lambda \ (\in (0,1])$ in $p_n^\lambda$ 
is decreased, i.e.,

     \begin{align}
       \label{relatn}  
    \mid{ t_j-c}\mid\propto\frac{1}{\lambda}.
     \end{align}
    To validate these assertions, the roots of $p_n^\lambda$ and the errors in approximation for various values of $\lambda$ have been graphically represented in the subsequent figures.
   \begin{figure}[ht]
\centering
\begin{minipage}{.5\textwidth}
  \centering
  \includegraphics[width=0.96\linewidth]{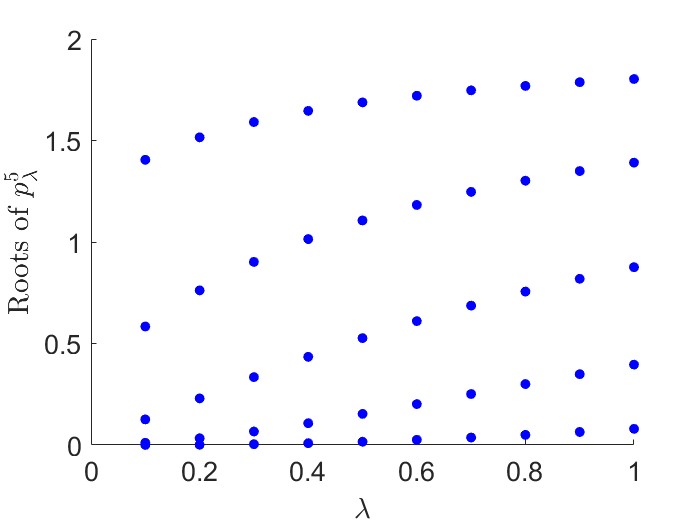}
  
  \label{fig:test1}
\end{minipage}%
\begin{minipage}{.5\textwidth}
  \centering
  \includegraphics[width=0.96\linewidth]{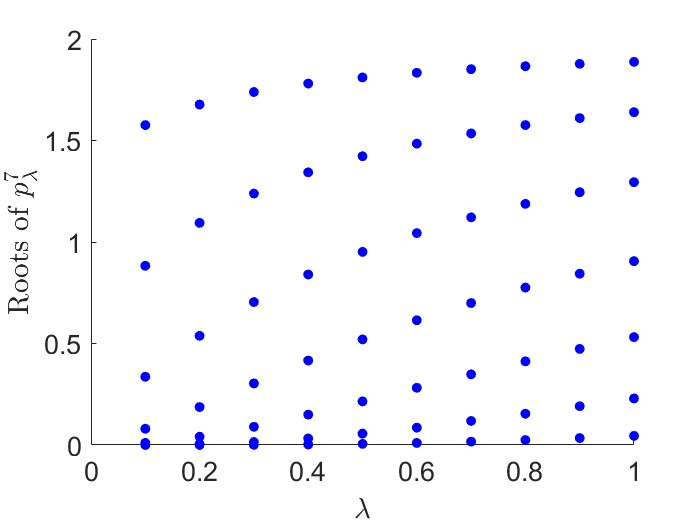}
 
  \label{fig:test2}
\end{minipage}
\caption{Distribution of roots of  orthogonal M\"untz polynomials of degree $5$ and $7$ for different values of $\lambda.$}
\label{figure 1}
\end{figure}\par
Figure \ref{figure 1} illustrates the distribution of zeroes of the orthogonal M\"untz polynomials of degree $5$ and $7$ with respect to the weight function $w'(t)=(2-t)$. It clearly shows how decreasing $\lambda$ shifts the zeroes of the orthogonal M\"untz polynomials far away from the root of the weight function $w'(t)$, i.e., $2$. Following this, Figure \ref{L2 norm} displays the $L_2$ norm of the distance of zeroes of $p_n^\lambda$ from the zero of the weight function $w'(t)=(2-t)^p~ (p=1,2,3)$, which confirms the relation (\ref{relatn}). It is also observed that increasing $p$ in the weight function causes the graph to become nearly linear, and the $L_2$ norm to increase, thereby making the zeroes more suitable to serve as nodes for improved approximation. This effect arises because, in the vicinity of $2$, the value of the weight function $(2-t)^p$ diminishes as $p$ increases. 
\begin{figure}[ht]
    \centering
    \includegraphics[width=0.7\linewidth]{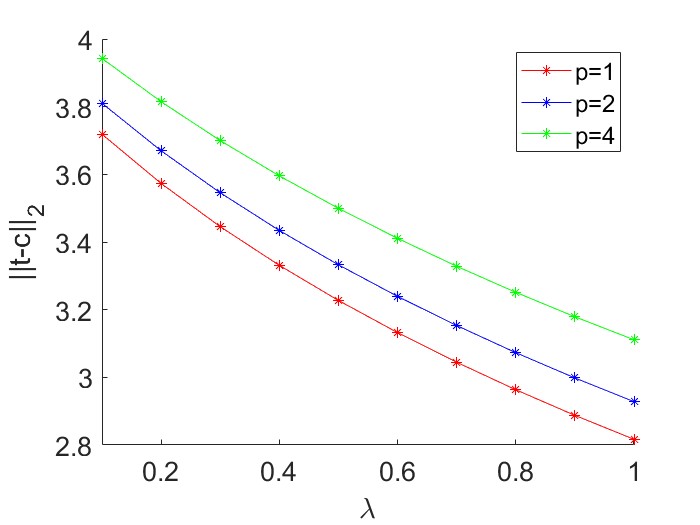}
    \caption{$L_2$ norm of  distance of the roots of orthogonal M\"untz  polynomial of degree $5$ from the zero of the weight function $(2-t)^p$. }
    \label{L2 norm}
\end{figure}

An example is provided to demonstrate the benefit of using the roots of the orthogonal M\"untz polynomial $p_n^{\lambda}$ with respect to $w'(t)=(c-t)$ for approximating the integral of a function that includes terms raised to a fractional power $\lambda$. Consider the function $p_{\lambda}(t)=\sin(3t^\lambda)$ to be integrated with respect to the weight function $w(t)=e^{-2t}$ over the interval [0,2]. Given that $w(t)$ approaches zero as $t\rightarrow2$, $p_{N+1}^{\lambda}$ is determined with respect to the weight function $w'(t)=(2-t)$ for various values of $\lambda$. The integral is then approximated using the fractional quadrature rule (\ref{4.5}) as
$$\int_0^2 e^{-2t}sin(3t^{\lambda}) \approx  \sum_{j=0}^Nsin(3t_j^{\lambda})w_j,  $$ 
where $ w_j=\int_0^2 h_j(t) e^{-2t} dt$ and $h_j(t)$ as defined in (\ref{hj}). In this case, $t_j$'s are the zeroes of $p_{N+1}^{\lambda}$.
  
 \begin{figure}[htbp]
     \centering
     \includegraphics[width=0.7\linewidth]{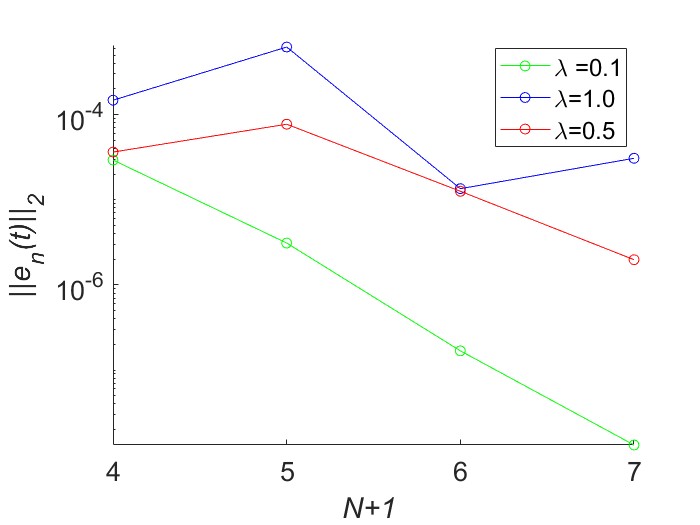}
     \caption{$L_2$ error in the approximation of the function $p_\lambda(t)=e^{-2x}sin(3t^{\lambda})$ with orthogonal M\"untz polynomial of degree $N+1$, using the fractional quadrature rule (\ref{4.5}) for different values of $\lambda$.
     }
     \label{fError in Approximation}
 \end{figure}
 Figure \ref{fError in Approximation} depicts well that when $\lambda$ is decreased, the approximation error also decreases significantly. 
 Hence, by considering these nodes in the fractional quadrature rule (\ref{4.5}), a more accurate approximation of the integral is achieved. It is concluded that the approximation of functions using this fractional quadrature rule proves to be significantly more advantageous in the case of functions composed of fractional powers, and also that the accuracy of the approximation is enhanced when the roots of the orthogonal M\"untz polynomials are employed as the nodes in the fractional quadrature rule. Consequently, the quadrature error, $R_N(p)$, in the fractional quadrature rule (\ref{4.5}) is reduced, making the rule (\ref{4.5}) more effective when the function $p_\lambda(t)$ is characterized by fractional power terms.
 
Now that the exactness of the fractional quadrature rule has been discussed, the bound of the quadrature error $R_N(p)$ will be determined in the following section.

\subsection{Error Analysis}
In any quadrature rule, there is an associated error in the approximation, which depends on the properties of the function to be approximated and the method itself. This section finds the bound of absolute error \cite{wu2023gaussian} that occurs when a function is approximated using the proposed rule (\ref{4.5}). For the same,  an expression is derived for the error bound in the following theorem.
\begin{thm}
    Let $R_N(p)$ be the remainder of the fractional quadrature rule defined in (\ref{4.5}) and let 
    $$k_w = \int_a^b w(t) dt,$$
    where $-1 \leq a, b \leq 1$.
    \begin{enumerate}[label=(\alph*)]
    \item If we consider an ellipse with foci $\pm 1$ such that the lengths of the semi-major and semi-minor axis sums to $r>1$ and $p$ is analytic in the region bounded by this ellipse. Moreover, if $\mid p(t) \mid \leq M_r$, then for each $N \geq 0$
    $$ \mid R_N(p) \mid \leq \frac{4k_w M_r}{r^{2N+2}(1-r^{-1})}.$$
    \item If $p, p^{(1)},\ldots, p^{(s-1)}$ are absolutely continuous in the closed interval $[-1,1]$ and $p^{(s)}$ is of bounded variation $V_s$ for some $s\in \mathbb N$. Then, for each $N\geq \lfloor \frac{s}{2}\rfloor $ 
    $$\mid R_N(p) \mid \leq \frac{4k_wV_s}{\pi s (2N+1)\ldots(2N-s+2)}.$$
    \end{enumerate}
\end{thm}
\begin{proof}
    First, to prove the part (a) of the theorem, we consider the Chebyshev expansion of  $p$,
        \begin{align}\label{cheb}
            p(t) = \frac{a_0}{2} + \sum_{j=1}^\infty a_j T_j(t),
        \end{align}
where \begin{align}
    a_j = \frac{2}{\pi} \int_{-1}^1 \frac{p(t)T_j(t)}{\sqrt{1-t^2}} dt.
\end{align}
Substituting the expansion (\ref{cheb}) in $R_N(p)$ and using  
Theorem (\ref{th 4.1}), we get
\begin{equation}\label{relat}
\begin{aligned}
 R_N(p) &= \sum_{j=2N+2}^\infty a_j E_N[T_j]\\
 &=\sum_{j=2N+2}^\infty a_j\left(I[T_j] - \sum_{k=0}^N w_k T_j(t_k)\right),
\end{aligned}
\end{equation}
where
\begin{align}
I[T_j] = \int_a^b T_j(t) w(t) dt. 
\end{align}
Since $\mid T_j(t)\mid \leq 1$ and the fractional quadrature rule is exact for polynomials of degree up to $2N+1$, therefore
\begin{equation}
\begin{aligned}\label{ineq}
\mid I[T_j] - \sum_{k=0}^N w_k T_j(t_k) \mid &\leq k_w + \sum_{k=0}^N w_k \\
& = 2k_w.
\end{aligned}
\end{equation}
Now, using \cite{MR4050406}, the Chebyshev coefficients $a_j$ satisfy the relation,
\begin{align}\label{coeff}
    \mid a_j \mid \leq \frac{2M_r}{r^j}, ~~~\forall j\geq 1.
\end{align}
Again using Equation (\ref{ineq}) and (\ref{coeff}), we get
\begin{align}
    \mid R_N(p) \mid &\leq \sum_{j=2N+2}^\infty 
    \frac{2M_r}{r^j}2k_w\\
    &= \frac{4k_wM_r}{r^{2N+2}(1-r^{-1})}.
    \end{align}
For proving the next part, again using \cite{MR4050406}, the Chebyshev coefficients $a_j$ for $j\geq s+1$ satisfy,
\begin{align}\label{Vp}
    \mid a_j \mid \leq \frac{2V_s}{\pi j (j-1) \ldots (j-m)}.
\end{align}
Using Equation (\ref{ineq}) and (\ref{Vp}) in (\ref{relat}),
\begin{align*}
    \mid R_N(p) \mid &\leq \sum_{j=2N+2}^\infty
    \frac{4k_wV_s}{\pi j (j-1)\ldots(j-s)}\\
    &=\frac{4k_wV_p}{\pi s} \sum_{j=2N+2}^\infty
    \left[\frac{1}{(j-1)\ldots(j-s)}-\frac{1}{j\ldots(j-s+1)}\right]\\
    &=\frac{4k_wV_s}{\pi s (2N+1)\ldots(2N-s+2)}.
    \end{align*}
\end{proof} 
The properties of the proposed fractional quadrature rule have been established, and its superior performance in approximating integrals involving fractional power functions has been demonstrated. In particular, it is easy to verify that when applied within the framework of M\"untz–Legendre scaling functions—which themselves contain fractional exponents—the fractional quadrature rule \eqref{4.5} provides significantly better accuracy than the classical quadrature method \cite{shen2011spectral}, due to its proven effectiveness in handling fractional power functions. The proposed fractional quadrature rule \eqref{4.5} is employed in the next section to evaluate the left Riemann–Liouville fractional integrals of the M\"untz–Legendre scaling functions.
 \section{Operational Matrix}\label{section4}
 In this section, we derive the operational matrix for the left Riemann–Liouville fractional integral of the M\"untz-Legendre scaling functions. To facilitate the computation, we first construct the operational matrix for the piecewise fractional power functions, which form the spanning set of the M\"untz-Legendre scaling functions. We then determine the corresponding transformation matrix, and by combining it with the operational matrix of the piecewise fractional power functions, we obtain the operational matrix for the M\"untz-Legendre scaling functions.
 \subsection{Fractional Integration of Piecewise Fractional Power  Functions }
The piecewise fractional power  functions on $[0,1)$ \cite{rahimkhani2018muntz} which are analogous to the M\"untz-Legendre scaling functions are defined as
\begin{align} \label{piece}
    \mathcal{T}_{n,m}^J(t)=\begin{cases}
			\left(t-\frac{n-1}{2^{J-1}}\right)^{\lambda_m}, & \text{if $\frac{n-1}{2^{J-1}}\leq t <\frac{n}{2^{J-1}}$ },\\
            0, & \text{otherwise},
		 \end{cases}
\end{align}
where $n=1,2, \ldots, 2^{J-1}$ is the translation parameter, $m=0,1,\ldots,M-1$ is the degree of the piecewise fractional power  polynomials and $\lambda_m=m\lambda$ ($\lambda$ is a real constant).\\
Let us write
\begin{align*}
    \textbf{T}(t)=&[\mathcal{T}_{1,0}^J(t),\mathcal{T}_{1,1}^J(t),\ldots,\mathcal{T}_{1,M-1}^J(t),\mathcal{T}_{2,0}^J(t),\ldots,\mathcal{T}_{2,M-1}^J(t),\ldots,\mathcal{T}_{2^{J-1},0}^J(t),\\
    &\ldots,\mathcal{T}_{2^{J-1},M-1}^J(t)]^T\\
    =&[\mathcal{T}_{1}(t),\mathcal{T}_2(t),\ldots,\mathcal{T}_M(t),\mathcal{T}_{M+1}(t),\ldots,\mathcal{T}_{2^{J-1}M}(t)]^T.
\end{align*}
The left Riemann-Liouville fractional integral of $\mathcal{T}_{n,m}^J(t)$ is given by
\begin{align}\label{integrall}
    _{0}I_{t}^\alpha\mathcal{T}_{n,m}^J(t)=\frac{1}{\Gamma(\alpha)}\int_0^t(t-z)^{\alpha-1}\mathcal{T}_{n,m}^J(z)dz.
\end{align}
To evaluate this integral, $t$ is considered in different subintervals, and the computation is carried out in three distinct cases.
\begin{case}{$\boldsymbol{0\leq t<\frac{n-1}{2^{J-1}}}$}\\
From the definition of $\mathcal{T}_{n,m}^J(t)$ in (\ref{piece}), it is concluded that the support of this function is the semi-closed interval $[\frac{n-1}{2^{J-1}},\frac{n}{2^{J-1}})$. Therefore, in this case, the value of this function becomes zero everywhere in the interval $(0,t)$, and hence,
\begin{align*}
 _{0}I_{t}^\alpha\mathcal{T}_{n,m}^J(t)=\frac{1}{\Gamma(\alpha)}\int_0^t(t-z)^{\alpha-1}\mathcal{T}_{n,m}^J(z)dz = 0.   
\end{align*}
    
\end{case}
\begin{case}{$\boldsymbol{\frac{n-1}{2^{J-1}}\leq t < \frac{n}{2^{J-1}}}$}\\
Again from (\ref{piece}), since the support of $\mathcal{T}_{n,m}^J(t)$ is the closed interval $[\frac{n-1}{2^{J-1}},\frac{n-1}{2^{J-1}})$. Hence, when ${\frac{n-1}{2^{J-1}}\leq t < \frac{n}{2^{J-1}}}$, the value of this function is non-zero only in $[{\frac{n-1}{2^{J-1}},t}]$ between the limits of \\ integration $0$ and $t$. Thus, (\ref{integrall}) can be written as
\begin{align*}
 _{0}I_{t}^\alpha\mathcal{T}_{n,m}^J(t)&=\frac{1}{\Gamma(\alpha)}\int_{\frac{n-1}{2^{J-1}}}^t(t-z)^{\alpha-1}\left(z-\frac{n-1}{2^{J-1}}\right)^{\lambda_m}dz\\
 &=\frac{1}{\Gamma(\alpha)}\int_{\frac{n-1}{2^{J-1}}}^t\left[t-\frac{n-1}{2^{J-1}}-\left(z-\frac{n-1}{2^{J-1}}\right)\right]^{\alpha-1}\left(z-\frac{n-1}{2^{J-1}}\right)^{\lambda_m}dz\\
 &=\frac{1}{\Gamma(\alpha)}\left(t-\frac{n-1}{2^{J-1}}\right)^{\alpha+\lambda_m-1}\int_{\frac{n-1}{2^{J-1}}}^t\Bigg\{\left(1-\frac{z-\frac{n-1}{2^{J-1}}}{t-\frac{n-1}{2^{J-1}}}\right)^{\alpha-1}\\&\quad
 \times\left(\frac{z-\frac{n-1}{2^{J-1}}}{t-\frac{n-1}{2^{J-1}}}\right)^{\lambda_m}\Bigg\}dz.
\end{align*}
Now putting $\left(\frac{z-\frac{n-1}{2^{J-1}}}{t-\frac{n-1}{2^{J-1}}}\right)=x$ and $dz=\left(t-\frac{n-1}{2^{J-1}}\right)\cdot dt$, the above integral becomes,
\begin{align*}
 _{0}I_{t}^\alpha\mathcal{T}_{n,m}^J(t)&= \frac{1}{\Gamma(\alpha)}  \left(t-\frac{n-1}{2^{J-1}}\right)^{\alpha+\lambda_m}\int_0^1 (1-x)^{\alpha-1} x^{\lambda_m}dt\\
 &=\frac{1}{\Gamma(\alpha)}  \left(t-\frac{n-1}{2^{J-1}}\right)^{\alpha+\lambda_m}\frac{\Gamma(\alpha)\Gamma(\lambda_m+1)}{\Gamma(\alpha+\lambda_m+1)}.
\end{align*}
    
\end{case}
\begin{case}${\boldsymbol{t\geq \frac{n}{2^{J-1}}}}$\\ \label{case3}
The limit of integration is $(0,t)$, but the function is non-zero in this interval only when\\ $\frac{n-1}{2^{J-1}} \leq t <\frac{n}{2^{J-1}}$. Thus the limit of integration can be changed to $\left(\frac{n-1}{2^{J-1}},\frac{n}{2^{J-1}}\right)$. Hence, in this case, the integral (\ref{integrall}) becomes
\begin{align}\label{definite intg}
\frac{1}{\Gamma(\alpha)}\int_{\frac{n-1}{2^{J-1}}}^{\frac{n}{2^{J-1}}}(t-z)^{\alpha-1}\left(z-\frac{n-1}{2^{J-1}}\right)^{\lambda_m}dz.    
\end{align}
Since both terms in the integrand product are raised to a fractional power, this definite integral cannot be solved directly using the exact integration techniques. Hence, the fractional quadrature rule (\ref{4.5}) is used to find the integral. Note that to achieve exactness in Equation \ref{definite intg}, we apply suitable substitutions to the integrand. This process correspondingly modifies the weight function. In this case, the nodes are taken as the roots of the Chebyshev polynomial of degree $M$ in the interval $\left[0,\frac{1}{2^{J-1}}\right]$. By employing the aforementioned fractional quadrature rule (\ref{4.5}), the integral (\ref{definite intg}) is transformed into
\begin{align*}
  _{0}I_{t}^\alpha\mathcal{T}_{n,m}^J(t)=\frac{1}{\Gamma(\alpha)}\sum_{j=0}^N t_j^{\lambda_m}w_j,
  \end{align*}
 where \quad $w_j=\int_{\frac{n-1}{2^{J-1}}}^{\frac{n-1}{2^{J-1}}} h_j(x)(t-x-\frac{n-1}{2^{J-1}})^{\alpha-1} dx $
  and  $h_j$ as defined in (\ref{hj}).
\end{case}
\begin{rmk}
    Since, exactness upto degree $M-1$ is sufficient for the exact computation of integrals involved in this work, the Chebyshev roots have been taken as nodes in the fractional quadrature rule (\ref{4.5}). One can also use the roots of orthogonal M\"untz polynomials as per their need of higher degree of exactness.
\end{rmk}
Using the three cases discussed above, the operational matrix $F(t,\alpha)$ for the fractional-order left Riemann-Liouville integration of piecewise fractional power functions is constructed in the following theorem.
 \begin{thm}
     The operational matrix $F(t,\alpha)$ for the fractional-order left Riemann-Liouville   integration of piecewise fractional power  functions is given by
     \begin{align}\label{opphi}
 _{0}I_{t}^\alpha\textbf{T}(t)\approx F(t,\alpha)\textbf{T}(t),        
     \end{align}
     where $F(t,\alpha)$ is the $2^{J-1}M\cross 2^{J-1}M $ upper triangular  matrix as follows
     \[F(t,\alpha) = [B_{i,j}]=
     \begin{pmatrix}
         B_{1,1} & B_{1,2} & \ldots & B_{1,2^{J-1}} \\  
0 & B_{2,2} & \ldots & B_{2,2^{J-1}} \\
\vdots & \vdots & \ddots & \vdots \\  
0 & 0 & \ldots & B_{2^{J-1},2^{J-1}}   
     \end{pmatrix}, \]
    
     and $B_{i,j}$ is an $M \cross M$ matrix such that
     $$B_{i,i}=diag[a_{1,1}(i),a_{2,2}(i)\ldots,a_{M,M}(i)],$$
     where $$a_{r,r}(i)=\frac{1}{\Gamma(\alpha)}\left(t-\frac{(i-1)}{2^{J-1}}\right)^\alpha \frac{\Gamma(\alpha)\Gamma(\lambda_{r-1}+1)}{\Gamma(\alpha+\lambda_{r-1}+1)},$$
  \\   and for $i<j$, $B_{i,j}$ is given by
  \[B_{i,j} =
  \begin{pmatrix}
      b_{1,1}(i) & 0 &\ldots & 0\\
      b_{2,1}(i) & 0 &\ldots & 0\\
      \vdots & \vdots &\dots &\vdots\\
      b_{M,1}(i) & 0 & \ldots & 0
      \end{pmatrix},
  \]
  where
  $$b_{s,1}(i)  = \frac{1}{\Gamma(\alpha)}\sum_{j=0}^N\left(t_j\right)^{\lambda_{s-1}}w_j.$$
 In this case, $w_j$ refers to the expression as defined in Case \ref{case3}.
 \end{thm}
\begin{proof}
 The integral (\ref{integrall}) for different values of $t\in [0,1)$ can be written as-
\begin{align}\label{integral}
    _{0}I_{t}^\alpha\mathcal{T}_{n,m}^J(t)\approx\begin{cases}
			0, &  0\leq t<\frac{n-1}{2^{J-1}}\\
            \frac{1}{\Gamma(\alpha)}  \left(t-\frac{n-1}{2^{J-1}}\right)^{\alpha+\lambda_m}\beta(\alpha,\lambda_{m}+1), &\frac{n-1}{2^{J-1}} \leq t <\frac{n}{2^{J-1}} \\
            \frac{1}{\Gamma(\alpha)}\sum_{j=0}^N\left(t_j\right)^{\lambda_m}w_j , & t \geq\frac{n}{2^{J-1}}.
		 \end{cases}
\end{align}
This can also be written as
\begin{align}
    _{0}I_{t}^\alpha\mathcal{T}_{n,m}^J(t)\approx\begin{cases}
			0, &  0\leq t<\frac{n-1}{2^{J-1}}\\
            \frac{1}{\Gamma(\alpha)}  \left(t-\frac{n-1}{2^{J-1}}\right)^{\alpha}\beta(\alpha,\lambda_{m}+1) \mathcal{T}_{n,m}(t), &\frac{n-1}{2^{J-1}} \leq t <\frac{n}{2^{J-1}} \\
            \frac{1}{\Gamma(\alpha)}\sum_{j=0}^N\left(t_j\right)^{\lambda_m}w_j \mathcal{T}_{n,0}(t) , & t \geq\frac{n}{2^{J-1}}.
		 \end{cases}
\end{align}
Treating the above expression as a linear combination of basis functions in $\textbf{T}(t)$, we get,
\begin{align*}
  _{0}I_{t}^\alpha\mathcal{T}(t) \approx F(t,\alpha) \mathcal{T}(t). 
\end{align*}
\end{proof} 
Using the operational matrix $F(t,\alpha)$ for the fractional-order left Riemann-Liouville integration of piecewise fractional power functions, the operational matrix $F(t,\alpha)$ for the fractional-order left Riemann-Liouville integration of M\"untz-Legendre scaling functions is found by determining the transformation matrix $\Omega$ in the subsequent subsection.
\subsection{Transformation Matrix}
The aim is to find the integration operational matrix for M\"untz-Legendre scaling functions, denoted as $\Phi$. To achieve this, a transformation matrix is identified, allowing for the derivation of the integration operational matrix for $\Phi$ using $F(t,\alpha)$. Let the transformation matrix of $\Phi$ to $\textbf{T}$ be denoted by $\Omega$, i.e.,
\begin{align}\label{11}
     \textbf{T}_{2^{J-1} M}(t)=\Omega_{{2^{J-1} M}\cross {2^{J-1} M}} \Phi_{2^{J-1} M}(t).
\end{align}
Equating the components of vectors on both sides,
\begin{align}
    \mathcal{T}_i(t)=\sum_{j=1}^K w_{ij} \phi_j(t), \ \ \ i=1,2,\ldots K,
    \end{align}
    where
    \begin{align}
        w_{ij}=\langle \mathcal{T}_i, \phi_j \rangle = \int_0^1 \mathcal{T}_i(t) \phi_j(t) dt, \ \ \ i=1,2,\ldots,K, \ \ \ j=1,2,\ldots,K,
    \end{align}
    and $\Omega=[w_{ij}]$ is a matrix of order ${2^{J-1} M}\cross {2^{J-1} M}$.
    \begin{ex}
        Let $\lambda=0.75.$ For $J=2$ and $M=3$, the transformation matrix is given by 
        \[
        \Omega=
        \begin{pmatrix}
          0.7071 &0 &0 &0 &0 &0 \\
          0.2403    &0.1140 &0 &0 &0 &0\\
          0.1000  &0.0730  &0.0173  &0  &0  &0\\
           0    &0  &0   &0.7071        &0         &0\\
           0  &0 &0    &0.2403    &0.1140         &0\\
           0 &0  &0    &0.1000    &0.0730   &0.0173
        \end{pmatrix}.
        \]
    \end{ex}
    In the next subsection, the operational matrix $P(t,\alpha)$ for the fractional-order left Riemann-Liouville integration of M\"untz-Legendre scaling functions is obtained using the transformation matrix $\Omega$ and the operational matrix $F(t,\alpha)$ for the fractional-order left Riemann-Liouville integration of piecewise fractional power functions.
\subsection{M\"untz-Legendre Scaling Functions Operational Matrix }
\begin{thm}
   The operational matrix $P(t,\alpha)$ for the fractional-order left Riemann-Liouville   integration of M\"untz-Legendre scaling functions is given by
     \begin{align}\label{operational matrix}
 _{0}I_{t}^\alpha\Phi(t)\approx P(t,\alpha)\Phi(t) ,       
     \end{align}
     where $P(t,\alpha)$ is the $2^{J-1}M\cross 2^{J-1}M $  matrix expressed as 
     \begin{align}\label{opmat}
         P(t,\alpha)=\Omega^{-1}F(t,\alpha) \Omega.
     \end{align}
\end{thm}
\begin{proof}
  The integration operational matrix of M\"untz-Legendre scaling functions is given by
 \begin{align}\label{111}
     _{0}I_{t}^\alpha\Phi(t)\approx P(t,\alpha)\Phi(t).  
 \end{align}
 Equation (\ref{11}) can be used to rewrite $\Phi(t)$ as
 \begin{equation*}
 \Phi(t) =\Omega^{-1} \mathcal{T}(t).
\end{equation*}
Taking the Riemann-Liouville integral of the above expression, we obtain
 \begin{align}\label{1st}
    _{0}I_{t}^\alpha\Phi(t) =\  _{0}I_{t}^\alpha\Omega^{-1}\mathcal{T}(t).  
 \end{align}
Since, $\Omega$ is a constant matrix, therefore from (\ref{opphi}), we get 
\begin{equation}\label{2nd}
    \begin{aligned}
        _{0}I_{t}^\alpha\Phi(t)&={\Omega^{-1}} _{0}I_{t}^\alpha\mathcal{T}(t)\\
        &\approx \Omega^{-1}F(t,\alpha)\mathcal{T}(t).
    \end{aligned}
\end{equation}
Hence, using Equation (\ref{1st}) and (\ref{2nd}), we obtain 
\begin{align*}
    P(t,\alpha)\Phi(t)&=P(t,\alpha)\Omega^{-1} \mathcal{T}(t)\\
    &\approx \Omega^{-1} F(t,\alpha) \mathcal{T}(t).
\end{align*}
Therefore,  the operational matrix $P(t,\alpha)$ for the fractional-order left \\
Riemann-Liouville integration of M\"untz-Legendre scaling functions is given by
\begin{align*}
   P(t,\alpha)= \Omega^{-1} F(t,\alpha) \Omega.  
\end{align*}
\end{proof}
In the following section, the numerical scheme for solving fractional differential equations have been derived using the integration operational matrix (\ref{opmat}). Moreover, the benefits of the proposed method for solving fractional differential equations are verified. Some test problems are selected to discuss the efficiency of the method, and these problems are solved using the collocation method \cite{kumar2021collocation}.
\section{Description of the Numerical Method and Test Examples}\label{section5}
In this section, we develop a new collocation method for solving the FDE \eqref{eq1prblm} using the M\"untz-Legendre scaling functions and the operational matrix derived in Section \ref{section4}. The Caputo derivative of $y(t)$ is approximated using the M\"untz-Legendre scaling functions as in (\ref{span}), i.e.,
\begin{align}\label{eqn1}
    {}_0^CD_t^{\alpha} y(t) \approx C^T\Phi(t).
\end{align}
where $C=[c_1,c_2,\ldots,c_{2^{J-1}M}]$. Again, using the operational matrix for fractional integration in (\ref{operational matrix}), we have
\begin{align*}
    {}_0I_t^{\alpha}{}_0^CD_t^{\alpha} y(t)&\approx C^T {}_0I_t^{\alpha} \Phi(t)\\
    &\approx C^TP(t,\alpha)\Phi(t).
\end{align*}
Thus, we can conclude 
\begin{align}\label{eqn2}
    y(t)\approx C^TP(t,\alpha)\Phi(t)+y(0).
\end{align}
Using (\ref{eqn1}) and (\ref{eqn2}), the FDE \eqref{eq1prblm} can be written as,
\begin{align*}
C^T\Phi(t)\approx f\big(t, C^TP(t,\alpha)\Phi(t)+y_0\big).
\end{align*}
Imposing equality in these equations at the collocation points, a system of $2^{J-1}M$ algebraic equations are obtained as
\begin{align}
  C^T\Phi(t_n)\approx f\big(t_n, C^TP(t_n,\alpha)\Phi(t_n)+y_0\big).  
\end{align}
Solving these algebraic equations, the values of $2^{J-1}M$ constants  $c_1$, $c_2$ \ldots, $c_{2^{J-1}M}$ can be calculated, which can be used to approximate $y(t)$ at any other point in $[t_0,t_f].$\par 

Now, we demonstrate the effectiveness of the proposed method through several test examples, highlighting in particular how the fractional quadrature rule enhances the accuracy and overall performance of the numerical method. The first example compares the accuracy of the proposed method with the Block-Pulse Method (BPM) \cite{kumar2021collocation,li2011numerical,babolian2008new} for a fixed parameter value $\lambda = 1$. The second example further strengthens the applicability of the fractional quadrature rule by examining its performance in approximating fractional power functions for different values of $\lambda$. The third example illustrates the solution of a non-linear FDE using the proposed approach. Finally, the last example presents a case in which the most accurate results are obtained when the order of the Caputo derivative $\alpha$ matches the parameter $\lambda$. In all examples, we employ the shifted Chebyshev roots of the second kind as collocation points and compute the $L_2$-error, defined as
$$
\begin{aligned}
 e_y=\left\|y(t)- \left(C^TP(t,\alpha)\Phi(t)+y_0\right)\right\|_2.
\end{aligned}
$$ 

\begin{ex}\label{ex1}
In this example, consider the FDE\begin{align*}
   { }_0^CD_t^\alpha y(t) +y(t) &=\frac{24 t^{4-\alpha}}{\Gamma(5-\alpha)}+t^4,~~~ t\in (0,1],\\
   y(0)&=0.
\end{align*}

   \begin{table*}
\centering
\begin{tabular}{@{}|l|c|c|c|c|c|@{}}
\hline
$J$ & $M$ & Using BPM
& Using the proposed method  
 \\
\hline
$2$    & 2 &1.3436e-01 & 3.2510e-02  \\
\hline
$2$    & 3 &3.6372e-02 & 1.3959e-03 \\
\hline
$2$    &4 & 1.1720e-02& 1.3899e-16 \\
\hline
$3$    &2  &3.2555e-02 & 8.2362e-03 \\
\hline
$3$    & 3  &6.9724e-03 &8.6785e-05  \\
\hline
$3$    & 4 & 2.9600e-03 & 8.0865e-17 \\
\hline
$4$    & 2  & 5.8852e-03 & 2.1317e-03 \\
\hline
$4$    & 3 &1.4276e-03 &5.4644e-06 \\
\hline
$4$    & 4  &7.4422e-04 & 2.3108e-16 \\

\hline
\end{tabular}
\caption{Error ($e_y$) in approximation for $\alpha=1$ and  $\lambda=1$ in Example \ref{ex1}.}
\label{table2}
\end{table*}
 
\end{ex}
The results of this example highlight the superior efficiency of the proposed method compared to existing techniques in the literature. The error in the approximation of $y(t)$ for  $\lambda=1$ is shown in Table \ref{table2}. It can be observed that the efficiency of the proposed method is better than the Block-Pulse method, and the error in approximation is decreased significantly in our method for small values of $J$ and $M$  when compared to the Block-Pulse method \cite{kumar2021collocation}.

In the next example, it is shown that how better approximations are obtained for different $\lambda$'s, where the choice of $\lambda$ depends on the function that is being approximated. Before proceeding to further examples, the presentation of the formula for the left Caputo derivative of the sine function, which will be utilized in subsequent discussions, is essential.
\begin{rmk}\label{capsin}
The left Caputo fractional derivative of the sine function is given \cite{nemati2018numerical} as follows.
$$
\begin{aligned}
{ }_0^CD_t^\alpha  \sin \lambda t & =-\frac{1}{2} i(i \lambda)^m t^{m-\alpha}\left(E_{1, m-\alpha+1}(i \lambda t)-(-1)^m E_{1, m-\alpha+1}(-i \lambda t)\right),
\end{aligned}
$$
where $E_{\alpha, \beta}(t)$ is the two-parameter Mittag-Leffler function given by
$$
E_{\alpha, \beta}(t)=\sum_{k=0}^{\infty} \frac{t^k}{\Gamma(\alpha k+\beta)}.
$$
\end{rmk}
\begin{ex}\label{ex3}
 In this example, consider the FDE
    \begin{align*}
         { }_0^CD_t^\alpha y(t) &= f(t),~~~ t\in (0,1],  \quad\text{where} \quad f(t)={ }_0^CD_t^\alpha sin(\pi t^{\frac{3}{2}}), \\
         y(0)&=0.
    \end{align*}
    

   
    \begin{table*}
    \centering
\begin{tabular}{@{}|l|c|c|c|c|c|@{}}
\hline
$J$ & $M$ & {$\lambda=1$}
&{$\lambda=0.75$} & {$\lambda=0.5$} & {$\lambda=0.25$}   
 \\
 \hline
2 &	2 &	1.8674e-01 &	2.2782e-01 &	2.9963e-01 &	4.2013e-01 \\	
\hline
2 &	3 &	3.2899e-02 &	2.7766e-02 &	4.8977e-02 &	2.0227e-01 \\	
\hline
2 &	4 &	1.2077e-02 &	8.3128e-03 &	3.6302e-02 &	1.6337e-01 \\	
\hline
3 &	2 &	5.2043e-02 &	6.3909e-02 &	8.7042e-02 &	1.2749e-01 \\	
\hline
3 &	3 &	8.6351e-03 &	5.7595e-03 &	1.3680e-02 &	6.3166e-02 \\	
\hline
3 &	4 &	4.5232e-03 &	2.2064e-03 &	3.9807e-03 &	3.4499e-02 \\	
\hline
4 &	2 &	1.6229e-02 &	2.0684e-02 &	3.0681e-02 &	4.8352e-02 \\	
\hline
4 &	3 &	2.9575e-03 &	1.8242e-03 &	3.6946e-03 &	2.0628e-02 \\	
\hline
4 &	4 &	1.6453e-03 &	8.4347e-04 &	6.6394e-04 &	9.3044e-03 \\	
\hline
\end{tabular}
\caption{Error ($e_y$) in approximation for $\alpha=1$ with the proposed method for different $\lambda$ in Example \ref{ex3}.}
\label{table4}
\end{table*}
\end{ex}
The error in approximating $y(t)$ across various $\lambda$ values is depicted in Table \ref{table4}, where the most accurate approximation is achieved with $\lambda=0.5$ and $0.75$. It can be noted that the method's precision is influenced by the specific functions being analyzed, particularly by the presence of fractional powers within these functions. This observation highlights the utility of the fractional quadrature rule in effectively approximating functions characterized by fractional powers.

Next, the method's applicability to a non-linear problem is demonstrated by analyzing a non-linear fractional initial value problem and its corresponding convergence results.
\begin{ex}\label{ex4}
  In this example, consider the non-linear FDE
    \begin{align*}
         { }_0^CD_t^\alpha y(t)+y^2(t) &= \frac{2t^{2-\alpha}}{\gamma(3-\alpha)}+t^4,~~~ t\in (0,1],   \\
         y(0)&=0.
    \end{align*}  
 \FloatBarrier      
    \begin{table*}[!h]
    \centering
\begin{tabular}{@{}|l|c|c|c|c|c|@{}}
\hline
$J$ & $M$ & {$\lambda=1$}
&{$\lambda=0.75$} & {$\lambda=0.5$} & {$\lambda=0.25$}   
 \\
\hline
2 &	2 &	3.0215e-15 &	1.6369e-02 &	4.0902e-02 &	7.9510e-02 \\	
\hline
2 &	3 &	2.0905e-15 &	2.0886e-03 &	1.5199e-15 &	2.0736e-02 \\	
\hline
2 &	4 &	5.0931e-15 &	6.5040e-04 &	2.0079e-13 &	2.6627e-03 \\	
\hline
3 &	2 &	2.4612e-15 &	7.2797e-03 &	1.7946e-02 &	3.4360e-02 \\	
\hline
3 &	3 &	6.1801e-15 &	8.6612e-04 &	2.0028e-15 &	8.9241e-03 \\	
\hline
3 &	4 &	4.5786e-16 &	2.8497e-04 &	2.6717e-13 &	1.1576e-03 \\	
\hline
4 &	2 &	4.3896e-12 &	3.4508e-03 &	8.4394e-03 &	1.6034e-02 \\	
\hline
4 &	3 &	7.3017e-16 &	3.9086e-04 &	3.2147e-13 &	4.1118e-03 \\	
\hline
4 &	4 &	1.1516e-12 &	1.3231e-04 &	1.1456e-14 &	5.3613e-04 \\	
\hline
\end{tabular}
\caption{Error ($e_y$) in approximation for $\alpha=1$ with the proposed method for different $\lambda$ in Example \ref{ex4}.}
\label{table5}
\end{table*}
\end{ex}
The accuracy in the approximation of $y(t)$ for various values of $\lambda$ is summarized in Table~\ref{table5}. Among the tested values, the most precise approximation is observed when $\lambda=1$. Furthermore, the example illustrates that the proposed method is also suitable for solving non-linear differential equations, as evidenced by the satisfactory approximation errors obtained.

In addition, we assess the performance of the proposed method on a linear fractional differential equation. The exact solution in this case is given as $y(t) = (t-0.5)^{4.5}$ for $t \geq 0.5$ and $y(t)=0$ otherwise, which is a four-times continuously differentiable function but not infinitely differentiable at $t=0.5$. This provides a good test case to evaluate the method's accuracy for solutions with finite smoothness.

\begin{ex}\label{ex5}
In this example, consider the FDE\begin{align*}
   { }_0^CD_t^\alpha y(t) +y(t) &=f(t),~~~ t\in (0,1],\\
   y(0)&=0,
\end{align*}
where
\begin{align*}
 f(t)=\begin{cases}
       \frac{\Gamma(5.5)}{\Gamma(5.5-\alpha)} \left(t-\frac{1}{2}\right)^{4.5-\alpha}+\left(t-\frac{1}{2}\right)^{4.5},\quad &\text{if} \quad t\geq \frac{1}{2},\\
     0, & \text{otherwise}.  
   \end{cases} 
   \end{align*}
   \FloatBarrier
  \begin{table*}[!h]
\centering
\begin{tabular}{@{}|l|c|c|c|c|c|@{}}
\hline
$J$ & $M$ & {$\lambda=1$}
&{$\lambda=0.75$} & {$\lambda=0.5$} & {$\lambda=0.25$}   
 \\
\hline
2 &	4 &	2.0157e-03 &	9.1071e-03 &	5.3559e-02 &	3.3320e-01 \\	
\hline
2 &	5 &	4.6870e-05 &	1.0077e-03 &	2.3087e-02 &	3.7748e-01 \\	
\hline
2 &	6 &	7.3746e-06 &	8.0108e-15 &	6.5484e-03 &	3.5321e-01 \\	
\hline
\end{tabular}
\caption{Error ($e_y$) in approximation for $\alpha=0.75$ with the proposed method for different $\lambda$ in Example \ref{ex5}.}
\label{table6}
\end{table*}
The numerical errors ($e_y$) for this example are presented in Table \ref{table6}, corresponding to a fractional order of $\alpha=0.75$. The table documents the errors for a fixed $J=2$ while varying the parameter $M$ and the newly introduced parameter $\lambda$. Two primary observations can be made. First, for any fixed value of $\lambda$, the error $e_y$ decreases rapidly as $M$ increases, which demonstrates the convergence of the method. Second, and more significantly, the accuracy is highly dependent on the choice of $\lambda$. The results clearly indicate that the optimal performance is achieved when the parameter $\lambda$ is chosen to be equal to the fractional order of the derivative, $\alpha$. In the case where $\lambda = \alpha = 0.75$, the error drops to the level of machine precision for $M=6$, confirming the method's high accuracy under this specific parameter selection. 
\end{ex}
\begin{rmk}
    Based on a thorough analysis of the proposed method applied to several examples, we highlight the following key observations, which may also be generalized to other numerical methods constructed using M\"untz--Legendre polynomials:
\begin{enumerate}
        \item The commonly held assumption that the best approximation is achieved when the order of the Caputo derivative ($\alpha$) matches the M\"untz parameter ($\lambda$) is not universally valid. Our numerical experiments indicate that this relationship is problem-dependent and influenced by the specific numerical scheme employed. Consequently, the optimal choice of $\lambda$ should be determined in a systematic or algorithmic manner, for instance, using the approach proposed in \cite{singh2023algorithm}. This observation is also supported by the literature, where several authors report improved results even when $\alpha \neq \lambda$ \cite{rahimkhani2018muntz,pourbabaee2022new,maleknejad2021numerical}. In many studies, the fractional order itself is not fixed but varies, such as in variable-order or distributed-order fractional models.
        
        \item As demonstrated in Example~\ref{ex5}, the proposed method yields highly accurate results for a special class of functions, including those with limited smoothness (i.e., functions that are not infinitely differentiable), provided an appropriate M\"untz parameter ($\lambda$) is selected. This indicates that the method is well-suited for such problems and can achieve excellent accuracy.
        
        \item A major contribution of the proposed approach is that it resolves a notable gap in the existing literature by eliminating the intermediate integration errors associated with the M\"untz scaling-function basis. To the best of our knowledge, this specific issue has not been adequately addressed in previous works.
    \end{enumerate}
\end{rmk}

\section{Conclusion} \label{section6}
Due to the distinct advantages offered by fractional orthogonal polynomials over their classical counterparts in solving mathematical models involving fractional derivatives and integrals, it is essential to develop fractional quadrature rules to give exact results for fractional polynomials. Therefore, this paper explores the collocation method employing M\"untz–Legendre scaling functions alongside the fractional quadrature rule for solving fractional differential equations, demonstrating satisfactory approximation results. The study investigates various aspects of the fractional quadrature rule, including its properties, such as exactness and the associated error bounds. The operational matrices associated with M\"untz–Legendre scaling functions have been derived, and using them, the fractional differential equations were transformed into a set of algebraic equations. To illustrate the method's practicality, three test problems were examined. For future work, the intention is to adapt this methodology to address various optimal control problems.
\subsection*{Data Availability Statement}
No data is available for this manuscript.
\subsection*{Conflict of Interest}
The author(s) declared no potential conflicts of interest with respect to the research, authorship, and/or publication of this article.
\bibliographystyle{elsarticle-num}
\bibliography{mybib}

\begin{thebibliography}{10}
\expandafter\ifx\csname url\endcsname\relax
  \def\url#1{\texttt{#1}}\fi
\expandafter\ifx\csname urlprefix\endcsname\relax\def\urlprefix{URL }\fi
\expandafter\ifx\csname href\endcsname\relax
  \def\href#1#2{#2} \def\path#1{#1}\fi

\bibitem{kumar2021collocation}
N.~Kumar, M.~Mehra, Collocation method for solving nonlinear fractional optimal control problems by using {H}ermite scaling function with error estimates, Optimal Control Applications and Methods 42~(2) (2021) 417--444.

\bibitem{MERAL2010939}
F.~Meral, T.~Royston, R.~Magin, Fractional calculus in viscoelasticity: An experimental study, Communications in Nonlinear Science and Numerical Simulation 15~(4) (2010) 939--945.

\bibitem{kulish2002application}
V.~V. Kulish, J.~L. Lage, Application of fractional calculus to fluid mechanics, J. Fluids Eng. 124~(3) (2002) 803--806.

\bibitem{assaleh2007modeling}
K.~Assaleh, W.~M. Ahmad, Modeling of speech signals using fractional calculus, in: 2007 9th International Symposium on Signal Processing and Its Applications, IEEE, 2007, pp. 1--4.

\bibitem{fellah2002application}
Z.~Fellah, C.~Depollier, M.~Fellah, Application of fractional calculus to the sound waves propagation in rigid porous materials: validation via ultrasonic measurements, Acta Acustica united with Acustica 88~(1) (2002) 34--39.

\bibitem{mathieu2003fractional}
B.~Mathieu, P.~Melchior, A.~Oustaloup, C.~Ceyral, Fractional differentiation for edge detection, Signal Processing 83~(11) (2003) 2421--2432.

\bibitem{sebaa2006application}
N.~Sebaa, Z.~E.~A. Fellah, W.~Lauriks, C.~Depollier, Application of fractional calculus to ultrasonic wave propagation in human cancellous bone, Signal Processing 86~(10) (2006) 2668--2677.

\bibitem{mehandiratta2019existence}
V.~Mehandiratta, M.~Mehra, G.~Leugering, Existence and uniqueness results for a nonlinear caputo fractional boundary value problem on a star graph, Journal of Mathematical Analysis and Applications 477~(2) (2019) 1243--1264.

\bibitem{canuto2007spectral1}
C.~Canuto, M.~Y. Hussaini, A.~Quarteroni, T.~A. Zang, Spectral methods: fundamentals in single domains, Springer Science \& Business Media, 2007.

\bibitem{canuto2007spectral2}
C.~Canuto, M.~Y. Hussaini, A.~Quarteroni, T.~A. Zang, Spectral methods: evolution to complex geometries and applications to fluid dynamics, Springer Science \& Business Media, 2007.

\bibitem{guo2006optimal}
B.-Y. Guo, J.~Shen, L.-L. Wang, Optimal spectral-galerkin methods using generalized jacobi polynomials, Journal of Scientific Computing 27 (2006) 305--322.

\bibitem{gottlieb1977numerical}
D.~Gottlieb, S.~A. Orszag, Numerical analysis of spectral methods: theory and applications, SIAM, 1977.

\bibitem{shen2011spectral}
J.~Shen, T.~Tang, L.-L. Wang, Spectral methods: algorithms, analysis and applications, Vol.~41, Springer Science \& Business Media, 2011.

\bibitem{pedas2012piecewise}
A.~Pedas, E.~Tamme, Piecewise polynomial collocation for linear boundary value problems of fractional differential equations, Journal of Computational and Applied Mathematics 236~(13) (2012) 3349--3359.

\bibitem{rahimkhani2018muntz}
P.~Rahimkhani, Y.~Ordokhani, E.~Babolian, M{\"u}ntz-legendre wavelet operational matrix of fractional-order integration and its applications for solving the fractional pantograph differential equations, Numerical algorithms 77 (2018) 1283--1305.

\bibitem{singh2023algorithm}
A.~K. Singh, M.~Mehra, An algorithm to estimate parameter in m{\"u}ntz-legendre polynomial approximation for the numerical solution of stochastic fractional integro-differential equation, Journal of Applied Mathematics and Computing (2023) 1--20.

\bibitem{kumar2023muntz}
N.~Kumar, M.~Mehra, M{\"u}ntz--legendre wavelet method for solving sturm--liouville fractional optimal control problem with error estimates, Mathematical Methods in the Applied Sciences (2023).

\bibitem{tripathi2013new}
M.~P. Tripathi, V.~K. Baranwal, R.~K. Pandey, O.~P. Singh, A new numerical algorithm to solve fractional differential equations based on operational matrix of generalized hat functions, Communications in Nonlinear Science and Numerical Simulation 18~(6) (2013) 1327--1340.

\bibitem{articleliterature1}
S.~Nemati, D.~F.~M. Torres, A new spectral method based on two classes of hat functions for solving systems of fractional differential equations and an application to respiratory syncytial virus infection, Soft Computing 25 (05 2021).

\bibitem{maleknejad2021numerical}
K.~Maleknejad, J.~Rashidinia, T.~Eftekhari, Numerical solutions of distributed order fractional differential equations in the time domain using the m{\"u}ntz--legendre wavelets approach, Numerical Methods for Partial Differential Equations 37~(1) (2021) 707--731.

\bibitem{bhrawy2015review}
A.~H. Bhrawy, T.~M. Taha, J.~A.~T. Machado, A review of operational matrices and spectral techniques for fractional calculus, Nonlinear Dynamics 81 (2015) 1023--1052.

\bibitem{li2015numerical}
C.~Li, F.~Zeng, Numerical methods for fractional calculus, Chapman and Hall/CRC, 2015.

\bibitem{akansu2010emerging}
A.~N. Akansu, W.~A. Serdijn, I.~W. Selesnick, Emerging applications of wavelets: A review, Physical communication 3~(1) (2010) 1--18.

\bibitem{mehra2018wavelets}
M.~Mehra, Wavelets Theory and Its Applications, Springer, 2018.

\bibitem{hosseinpour2018new}
S.~Hosseinpour, A.~Nazemi, E.~Tohidi, A new approach for solving a class of delay fractional partial differential equations, Mediterranean Journal of Mathematics 15 (2018) 1--20.

\bibitem{singh2024modified}
A.~K. Singh, M.~Mehra, A.~A. Alikhanov, Modified least squares method and a review of its applications in machine learning and fractional differential/integral equations (2024).
\newblock \href {http://arxiv.org/abs/2405.00382} {\path{arXiv:2405.00382}}.

\bibitem{wu2023gaussian}
M.~Wu, H.~Wang, Gaussian quadrature rules for composite highly oscillatory integrals, Mathematics of Computation (2023).

\bibitem{MR4050406}
L.~N. Trefethen, Approximation theory and approximation practice, extended Edition, Society for Industrial and Applied Mathematics (SIAM), Philadelphia, PA, [2020] \copyright 2020.

\bibitem{li2011numerical}
Y.~Li, N.~Sun, Numerical solution of fractional differential equations using the generalized block pulse operational matrix, Computers \& Mathematics with Applications 62~(3) (2011) 1046--1054.

\bibitem{babolian2008new}
E.~Babolian, Z.~Masouri, S.~Hatamzadeh, New direct method to solve nonlinear volterra-fredholm integral and integro-differential equations using operational matrix with block-pulse functions, Progress In Electromagnetics Research B 8 (2008) 59--76.

\bibitem{nemati2018numerical}
A.~Nemati, Numerical solution of 2d fractional optimal control problems by the spectral method along with bernstein operational matrix, International Journal of control 91~(12) (2018) 2632--2645.

\bibitem{pourbabaee2022new}
M.~Pourbabaee, A.~Saadatmandi, A new operational matrix based on m{\"u}ntz--legendre polynomials for solving distributed order fractional differential equations, Mathematics and Computers in Simulation 194 (2022) 210--235.

\end{thebibliography}

\end{document}